\documentclass{article}
\usepackage{latexsym, amssymb, enumerate, amsmath,amsthm,pdfsync, multicol}
\usepackage{fullpage}
\usepackage[utf8]{inputenc}     
\usepackage[T1]{fontenc}
\usepackage{lmodern}
\usepackage{graphicx}
\usepackage[english]{babel}
\usepackage[numbers]{natbib}
\usepackage{url}
\usepackage{esint}
\usepackage{dsfont}
\usepackage{float}
\usepackage{epsfig, epstopdf}
\usepackage{subfigure}
\usepackage{cleveref}
\usepackage[normalem]{ulem}
\usepackage{tikz}
\usetikzlibrary{patterns}

\usepackage{todonotes}

\crefname{hyp}{hypothesis}{hypotheses}
\crefname{thm}{theorem}{theorems}
\crefname{lem}{lemma}{lemmas}
\crefname{cor}{corollary}{corollaries}
\crefname{prop}{proposition}{propositions}
\Crefname{theorem}{Theorem}{Theorems}

\newcommand{\R}{\mathbb{R}}

\newcommand{\eps}{\epsilon}

\newcommand{\1}{\mathds{1}}

\newcommand{\ds}{\displaystyle}

\renewcommand{\epsilon}{\varepsilon}

\newcommand{\undu}{\underline{u}}

\newtheorem{thm}{Theorem}[section]
\newtheorem{prop}[thm]{Proposition}

\newtheorem{claim}[thm]{Claim}
\newtheorem{lemma}[thm]{Lemma}
\newtheorem{hyp}[thm]{Hypothesis}

\newtheorem{theorem}[thm]{Theorem}

\def\eps{{\varepsilon}}

\def\j{{\mathcal{J}}}

\def\D{{\mathcal{D}}}
\newcommand{\opd}[1]{\D[\,{#1}\,]}

\author{Emeric Bouin \footnote{CEREMADE - Universit\'e Paris-Dauphine, UMR CNRS 7534, Place du Mar\'echal de Lattre de Tassigny, 75775 Paris Cedex 16, France. E-mail: \texttt{bouin@ceremade.dauphine.fr}}
\and Jérôme Coville \footnote{UR 546 Biostatistique et Processus  Spatiaux, INRA, Domaine St Paul Site Agroparc, F-84000 Avignon, France. E-mail: \texttt{jerome.coville@inra.fr}}
\and Guillaume Legendre \footnote{CEREMADE - Universit\'e Paris-Dauphine, UMR CNRS 7534, Place du Mar\'echal de Lattre de Tassigny, 75775 Paris Cedex 16, France. E-mail: \texttt{legendre@ceremade.dauphine.fr}}}

\begin{document}
\title{Acceleration in integro-differential combustion equations}
\maketitle

\begin{abstract}
We study acceleration phenomena in monostable integro-differential equations with ignition nonlinearity. Our results cover fractional Laplace operators and standard convolutions in a unified way, which is also a contribution of this paper. To achieve this, we construct a sub-solution that captures the expected dynamics of the accelerating solution, and this is here the main difficulty. This study involves the flattening effect occurring in accelerated propagation phenomena.
\end{abstract}


\section{Introduction.}\label{sec:Intro}

In this paper, we are interested in describing quantitatively propagation phenomena in the following (non-local) integro-differential equation:
\begin{align}\label{eq:main}
&u_t(t,x) =\opd{u}(t,x) + f(u(t,x)) \quad \text{ for } \quad t>0, x\in\R,\\
&u(0,\cdot)=\mathds{1}_{(-\infty,0]},\nonumber
\end{align}
where $$\opd{u}(t,x):=P.V.\left(\int_{\R}[u(t,y)-u(t,x)] J(x-y)\,dy\right)$$
with $J$ is a nonnegative function satisfying the following properties.
\begin{hyp}\label{hyp:J}
Let $s \in [0,\frac12]$. The kernel $J$ is symmetric and is such that there exists positive constants $\j_0,\j_1$ and $R_0\ge 1$  such that 
\begin{equation*}
\int_{|z|\le 1}J(z)|z|^2\,dz\le 2\j_1 \quad \text{ and }\quad  \frac{\j_0}{|z|^{1+2s}}\1_{\{|z| \ge 1\} }\ge J(z)\ge \frac{\j_0^{-1}}{|z|^{1+2s}}\1_{\{|z|\ge R_0\} }.
\end{equation*}
\end{hyp}
The operator $\opd{}$ describes the dispersion process of the individuals. Roughly, the kernel $J$ gives the probability of a jump from a position $x$ to a position $y$, so that the tails of $J$ are of crucial importance to quantify the dynamics of the population. As a matter of fact, the parameter $s$ will thus appear in the rates we obtain later. One may readily notice that our hypothesis on $J$ allows to cover the two broad types of integro-differential operators $\opd{u}$ usually considered in the literature which are the fractional laplacian $(-\Delta)^s u$ and the standard convolution operators with integrable kernels often written $J \star u - u$. This universality is one main contribution of this paper. 

\begin{hyp}\label{hyp:f}
Take $\theta > 0$, 
\begin{equation*}
f(1) = 0, \qquad  f(u) = 0 \quad \text{if } u \leq \theta \qquad \text{and } f(u) > 0 \quad \text{in } [\theta,1].
\end{equation*} 
\end{hyp}

The strong maximum principle implies that the solution to \eqref{eq:main} takes values in $[0,1]$ only. Moreover, since the initial data is decreasing, at all times $t\in \R^+$, the function $x \mapsto u(t,x)$ is decreasing over $\R$, from one to zero. To follow the propagation, we may thus follow level sets of height $\lambda \in (0,1)$, 
\begin{equation*}
x_\lambda(t) := \sup\left\{ x \in \R, \, u(t,x) \geq \lambda \right\}.
\end{equation*}
Our main result is the following. 

\begin{theorem}\label{thm:main-igni}
Assume that $J$ satisfies \Cref{hyp:J} with $s < \frac12$ and that $f$ is an ignition nonlinearity. For any $\lambda\in(0,1)$, the level line $x_\lambda(t)$ accelerates with the following rate, 
\begin{equation*}
t^{\frac{1}{2s}} \lesssim x_\lambda(t) \lesssim t^{\frac{1}{2s}+0}.
\end{equation*}
\end{theorem}

The main purpose of this research report is to show the lower bound, the sharp upper bound will come in a later work. 

Let us review some existing works around this issue. Existence of fronts have been obtained for the fractional laplacian for $s>1/2$ by Mellet et al in \cite{Mellet2014} and for the convolution type operator $J\star u -u$ provided $J$ has a first moment by Coville \cite{Coville2007d,Chen1997}. See also some related works by Shen et al \cite{Shen2017,Shen2017a}. Here we explore a situation where no first moment at infinity exists, that is $s\le 1/2$. Thanks to the monostable results \cite{Coville2021,Gui2015}, we already know that accelerated propagation can only occur in this region of parameter. This is drastic contrast with truly monostable nonlinearities \cite{Cabre2009,Cabre2013}.


\begin{figure}[h]\centering
\begin{tikzpicture}[scale=0.8]
%
\draw[->,line width=0.2, dashed](-8,-0.1)--(8,-0.1) node[below]{$x$};
\draw [->,line width=2, blue] (0.1,2.5)--(2.5,2.5) node[anchor=south]{$\asymp t^{\frac{1}{2s}}$};

\draw[line width=1, dashed](-6,-1.5)--(-6,6.5) node[right]{$$};
\draw[line width=1, dashed](6,-1.5)--(6,6.5) node[right]{$$};

%
%
\draw [domain=-7.5:7.5,line width=3, samples=150, red] plot (\x, {(5*exp{-\x})/(1+exp{-\x})}) ;
\fill[blue] (0,2.5) circle (5pt);
\draw[<->,line width=1](-5.2,1)--(5.2,1) node[below]{$\sim t^{\frac{1}{2s}}$};


\end{tikzpicture}
\caption{Schematic view of the expected behaviour of solution at a given time $t$.}
\label{fig:view}
\end{figure}
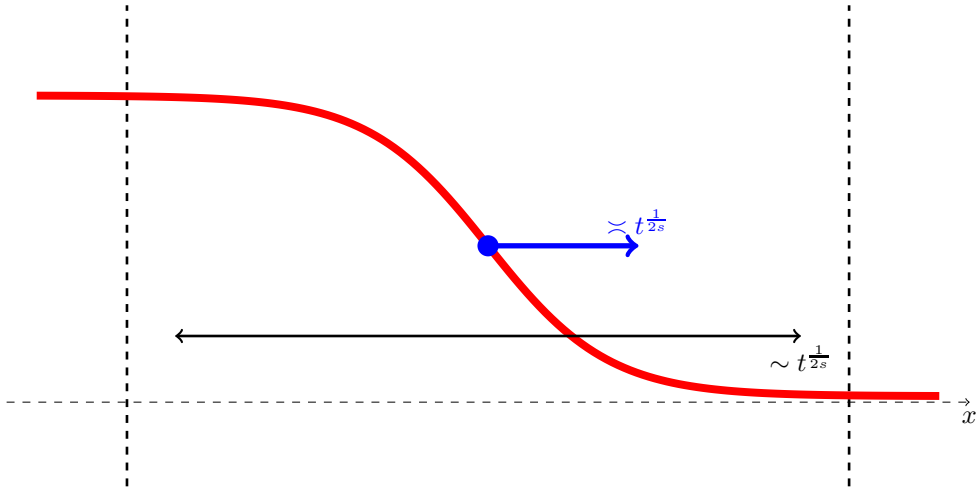


The rest of the paper is organised as follows. The following \Cref{sec:flatt} is about the behaviour of the linear problem. Then, \Cref{sec:strategy} describes in broad lines the construction of the sub-solution.

\section{Flattening in the linear problem}\label{sec:flatt}

Finally, the flattening behaviour of linear anomalous diffusions is reminiscent from \cite{Blumenthal1960,Kolokoltsov2000,Chasseigne2006,Bogdan2009}, and we shall here recall and extend some principal computations from these papers. 

The Fourier symbol of the operator $\opd{\cdot}$ is 
\begin{equation*}
\forall \xi \in \R, \qquad W(\xi) := \int_{\R} \left( \cos(\xi y) - 1 \right) J(y) \, dy.
\end{equation*}
Note that one may recover two typical cases. If $J(y) \propto \vert y \vert^{-1-2s}$, that is $\opd{\cdot}$ is a fractional Laplacian, then $W(\xi)= \vert \xi \vert^{2s}$. If $J$ is an integrable function with unit mass, as in convolution models, then $W(\xi)= \hat J(\xi) - 1$.
The presence of a singularity at $0$ for $J$ has an influence on large frequencies $\xi$, whereas the tail of $J$ influences small frequencies. As a consequence, 

For small $\xi$, write, 
\begin{equation*}
\int_{\R} \left( \cos(\xi y) - 1 \right) J(y) \, dy = \int_{\vert y \vert \leq R_0} \left( \cos(\xi y) - 1 \right) J(y) \, dy + \int_{\vert y \vert \leq R_0} \left( \cos(\xi y) - 1 \right) J(y) \, dy
\end{equation*}
The first integral in the r.h.s is of order $\vert \xi \vert^2$ with explicit constant by a direct Taylor expansion and using the hypothesis on $J$. The second one is estimated as follows. Since
\begin{equation*}
\int_{\vert y \vert \leq R_0} \j_0^{-1} \frac{\cos(\xi y) - 1}{\vert y \vert^{1+2s}} \, dy \leq \int_{\vert y \vert \leq R_0} \left( \cos(\xi y) - 1 \right) J(y) \, dy \leq \int_{\vert y \vert \leq R_0} \j_0 \frac{\cos(\xi y) - 1}{\vert y \vert^{1+2s}} \, dy,
\end{equation*}
we have that $\int_{\vert y \vert \leq R_0} \left( \cos(\xi y) - 1 \right) J(y) \, dy$ is of order $- \vert \xi \vert^{2s}$ with explicit estimates. As a consequence, since $s \leq \frac12$, $W$ is of order $- \vert \xi \vert^{2s}$ with explicit estimates.

In this section, we discuss the fact that the solution to the following linearised problem 
\begin{align}
&G_t =\opd{G}  \quad \text{ for } \quad t>0, x\in\R,\\
&G(0,\cdot)= \delta_{x=0},\nonumber
\end{align}
flattens, that is, there exists $C_0 \in \R^+$ such that
\begin{equation*}
\forall t \in \R^+, \exists x_0 \in \R^+, \qquad G(t,x) \geq \frac{C_0 t}{\vert x \vert^{1+2s}}.
\end{equation*}
The proof takes its spirit in \cite[Proposition 2.2]{Kolokoltsov2000} and we explain the arguments below. Note that then the solution to 
\begin{align}
&v_t =\opd{v}  \quad \text{ for } \quad t>0, x\in\R,\\
&v(0,\cdot)=\mathds{1}_{(-\infty,0]},\nonumber
\end{align}
is then given by 
\begin{equation*}
v(t,x) = G(t,\cdot) \star \mathds{1}_{(-\infty,0]}(\cdot) (x) = \int_x^{+\infty} G(t,y) \, dy. 
\end{equation*}
From this computation, observe that for any $t$, we have $\lim_{x \to -\infty} v(t,x) = 1$. Getting an estimate of $G$ for very large $y$ will yield an estimate for $v$.

Solving in the Fourier variable, this is 
\begin{equation*}
\forall \xi \in \R, \qquad \hat G(t,\xi) = \exp\left( W(\xi) t \right).
\end{equation*}
Observe that, formally,
\begin{equation*}
G(t,x) = \int_{\R} \exp\left( W(\xi) t - i x \xi \right) \, d \xi = 2 \text{Re}\left( \int_{0}^\infty \exp\left( W(\xi) t - i x \xi \right) \, d \xi \right).
\end{equation*}
Compute, for any $x \neq 0$,  
\begin{align*}
\int_0^R \exp\left( W(\xi) t - i x \xi \right) \, d\xi &= \frac{1}{x} \int_0^{Rx} \exp\left( W\left( \frac{u}{x} \right)t - i u \right) \, du
\end{align*}
Following the same steps as in \cite{}, we shall use a contour integral in the complex plane. Define the holomorphic function $\varphi(z) := \exp\left( W\left( \frac{z}{ix} \right)t - z \right)$ that we integrate on the contour ... . The Cauchy integral theorem gives 
\begin{align*}
&\frac{1}{x} \int_0^{Rx} \exp\left( W\left( \frac{u}{x} \right)t - i u \right) \, du\\
&= \frac{1}{ix} \int_0^{Rx} \exp\left( W\left( \frac{u}{ix} \right)t - u \right) \, du + \frac{1}{ix} \int_0^{\frac{\pi}{2}} \exp\left( W\left( \frac{Rx e^{i\theta}}{ix} \right)t - Rx e^{i\theta}  \right) \, i Rx e^{i\theta} d\theta\\ 
&= \frac{1}{ix} \int_0^{Rx} \exp\left( W\left( \frac{u}{ix} \right)t - u \right) \, du +  \int_0^{\frac{\pi}{2}} \exp\left( W\left( R e^{i\theta-i\frac{\pi}{2}} \right)t - Rx e^{i\theta}  \right) \, R e^{i\theta} d\theta
\end{align*}
The second integral goes to zero by the Lebesgue dominated convergence theorem. Taking $R \to \infty$, we get that the following integral exists and that 
\begin{align*}
2 \text{Re}\left( \int_0^\infty \exp\left( W(\xi) t - i x \xi \right) \, d\xi \right)
&= -\frac{2}{x} \text{Re}\left( i \int_0^{\infty} \exp\left( W\left( \frac{u}{ix} \right)t - u \right) \, du \right)\\
&= \frac{2}{x} \int_0^{\infty} \text{Im}\left( \exp\left( W\left( \frac{u}{ix} \right)t \right) \right) e^{-u} \, du\\
&= \frac{2}{x} \int_0^{\infty}  \exp\left( \text{Im}\left(W\left( \frac{u}{ix} \right)t \right)\right) \sin\left( \text{Im}\left(W\left( \frac{u}{ix} \right)t \right)\right) e^{-u} \, du 
\end{align*}
From the latter, we deduce that 
\begin{align*}
&\lim_{ x\to \infty} \left( \frac{2x^{1+2s} }{t}  \text{Re}\left( \int_0^\infty \exp\left( W(\xi) t - i x \xi \right) \, d\xi \right) \right)\\
&= \lim_{ x\to \infty} \frac{2}{t} x^{2s} \int_0^{\infty}  \exp\left( \text{Im}\left(W\left( \frac{u}{ix} \right)t \right)\right) \sin\left( \text{Im}\left(W\left( \frac{u}{ix} \right)t \right)\right) e^{-u} \, du\\
&\gtrsim 2 \int_0^{\infty} \sin(s\pi) u^{2s} e^{-u} \, du
\end{align*}
using the scaling of $W$ near zero and the dominated convergence theorem. 

\section{The strategy for the construction of sub-solutions.}\label{sec:strategy}

In this section, we present the way that we construct a sub-solution to prove the lower bounds in \Cref{thm:main-igni}. We are looking for a sub-solution $\undu$ to \eqref{eq:main} that satisfies everywhere
\begin{equation}\label{bcl2-eq:subsol-igni} 
\undu_t \leq \opd{\undu} + f(\undu) \qquad \text{and} \qquad \undu \leq \eps, 
\end{equation}
for some $\eps \in (\theta,1)$ and $t>t^*$. We construct an at least of class $\mathcal{C}^2$ function $\undu$ piecewise,
\begin{align*}
&\undu := \eps, \qquad \text{on } \left\{ x \leq X(t) \right\},\\
&\undu := \phi, \qquad \text{else},
\end{align*}
with $\phi(t,X(t))=\eps$. The point $X(t)$ is unknown at that stage. As previously for the monostable case, we expect $\phi$  to look like a solution of the standard fractional Laplace equation with Heaviside initial data at the far edge. In this situation, a natural candidate would be given by 
\begin{equation}\label{bcl2-def:w-igni}
w(t,x):= \left[\frac{x^{2s}}{\kappa t} +\gamma\right]^{-1}.
\end{equation}
with $\kappa,\gamma$ positive free parameter that will be determined later on.
Note that this function is well defined for $t\ge 1$ and $x>0$. 
The expected decay in space of a solution of the standard fractional Laplace equation with Heaviside initial data being at least of order $tx^{-2s}$, such a $w$ would have the good asymptotics. 
For $\eps\in (0,1)$, let us define $X(t)>0$ such that $w(t,X(t))=\eps$. For such $X(t)$ to be well defined, we need to  impose that $\gamma<\frac{1}{\eps}$, and thus  for such $\eps$ and $\gamma<\frac{1}{\eps}$, $X(t)$ is then defined by the following formula 
\begin{equation}\label{bcl2-def:X-igni}
X(t)=\left[\eps^{-1} -\gamma\right]^{\frac{1}{2s}}(\kappa t)^{\frac{1}{2s}}.
\end{equation}
One may observe that $X(t)$ moves with the speed that we expect in \Cref{thm:main-igni}. However, taking $\phi$ as this $w$ would not lead to a $\mathcal{C}^2$ function at $x=X(t)$. As in the monostable case,  to remedy this issue, we complete our construction by taking $\phi$ such that
\begin{equation}\label{bcl2-def:undu-igni}
\undu(t,x):=\begin{cases} \eps &\quad\text{ for all } x\le X(t),\medskip\\ 
\ds{3\left(1-\frac{1}{\eps}w(t,x)+\frac{1}{3\eps^2}w^{2}(t,x)\right)w(t,x)} &\quad\text{ for all } x>X(t),
\end{cases}
\end{equation}
for $t > 1$.

\section{Proof of \Cref{thm:main-igni}.}\label{sec:Proof}

Start by observing that $\undu$ defined in \eqref{bcl2-def:undu-igni} satisfies \eqref{bcl2-eq:subsol-igni} if and only if, 
\begin{align}
0 &\leq \opd{\undu} + f(\eps), &x \leq X(t),\label{eq:subsolleft-igni}\\
\undu_t &\leq  \opd{\undu}  + f(\undu), & \text{else}\label{eq:subsolright-igni} .
\end{align}
As a consequence, again the main work is to derive good estimates for $\opd{\undu}$ in both regions $x \leq X(t)$ and $x \geq X(t)$. The estimate in the first region will be rather direct to get and will rely mostly on the fact that $\undu$ is constant there together with the tails of $J$. In the latter region, things are more intricate. We have to split it into two zones, as depicted on \Cref{fig:zones-igni} below, each one been the stage of one specific character of the model and thus demanding a specific way to estimate $\opd{\undu}$. 


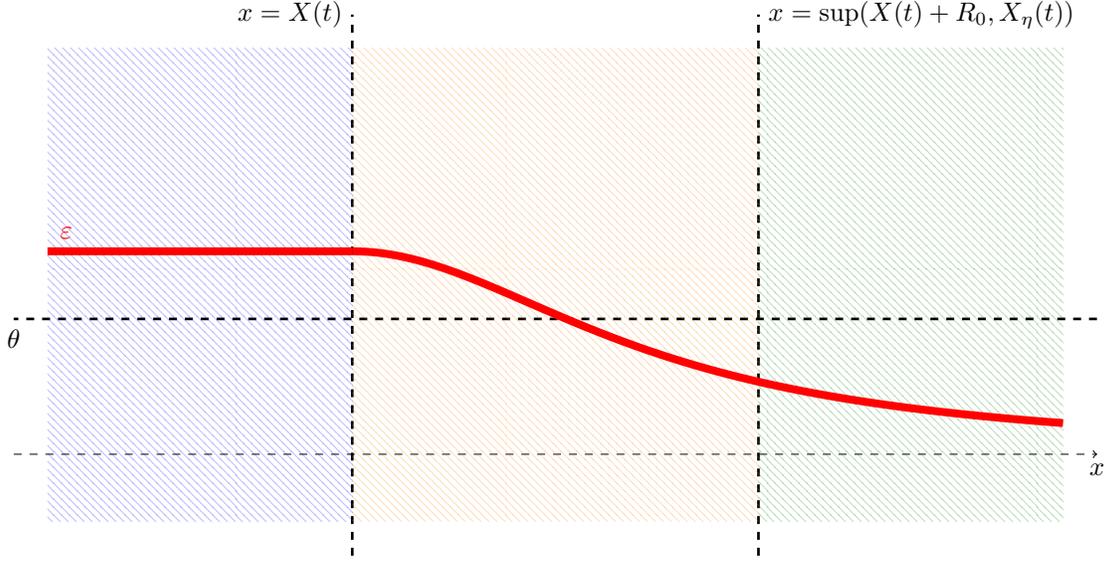
\begin{figure}[h]\centering
\begin{tikzpicture}[scale=0.9]
%
\draw[->,line width=0.2, dashed](-8,0)--(8,0) node[below]{$x$};
\draw[line width=1, dashed](8,2)--(-8,2) node[below]{$\theta$};

\draw[line width=1, dashed](-3,-1.5)--(-3,6.5) node[left]{$x=X(t)$};
\draw[line width=1, dashed](3,-1.5)--(3,6.5) node[right]{$x= \sup(X(t)+R_0,X_\eta(t))$};
%
\fill[pattern=north west lines, pattern color=blue, opacity=0.5] 
(-7.5,-1)--(-3,-1)--(-3,6)
-- (-7.5,6)
-- cycle;

\fill[pattern=north west lines, pattern color=orange, opacity=0.5] 
(-3,-1)--(-3,6) -- (0,6)--(0,-1)
-- cycle;

\fill[pattern=north west lines, pattern color=orange,opacity=0.5] 
(0,-1)--(0,6)--(3,6)--(3,-1)-- cycle;


\fill[pattern=north west lines,pattern color=black!50!green, opacity=0.5] 
(3,-1)--(3,6)--(7.5,6)--(7.5,-1)-- cycle;

%
%
\draw [domain=-3:7.5,line width=3, samples=150, red] plot (\x, {3/(1+0.05*(\x+3)^2)}) ;
\draw [line width=3,red] (-3,3)--(-7.5,3) node[anchor=south west]{$\varepsilon $};
\end{tikzpicture}
\caption{Schematic view of the sub-solution at a given time $t$. Several zones have to be considered. The exact expression of $X_\eta(t)$ will appear naturally later. The blue zone is where $\undu$ is constant, making computations easier. In the orange zone, the fact that $\undu >\theta$  is crucial. In the green (far-field) zone, the decay imitating a fractional Laplace equation gives the right behaviour.}
\label{fig:zones-igni}
\end{figure}


Let us now show that for the right choice of $\eps$ and $\kappa$ the function $\undu$ is indeed a subsolution to \eqref{bcl2-eq:subsol-igni} for all $t\ge 1$.

\subsection{Some preliminary estimates.}

\subsubsection{Facts and formulas on $X$ and $w$.}
As in the previous construction, let us recall some useful facts.
From direct computations we have:
\begin{align}
&\undu_t =   \undu_x = \undu_{xx} =0 &\; \text{ for all } \; t>0, x<X(t)\label{bcl2-def:u_tx-igni}\\
&\undu_t = 3 w_t \left(1- \frac{w}{\eps}  \right)^2, \qquad \undu_x = 3 w_x \left(1- \frac{w}{\eps}  \right)^2,&\; \text{ for all } \; t>1, x>X(t)\label{bcl2-def:u_t-u_x-igni}\\
&\undu_{xx}(t,x)= 3\left(1-\frac{w }{\eps}\right)\left[w_{xx} \left(1- \frac{w}{\eps} \right)-\frac{2 w_x^2}{\eps}\right]&\; \text{ for all } \; t>1, x>X(t)\label{bcl2-def:u_xx-igni}
\end{align}
Note crucially that $\undu$ is then at a $\mathcal{C}^{2}$ function in $x$ and $\mathcal{C}^1$ in $t$. 
We will also need repeatedly the following information on derivatives of $w$ at any point $(t,x)$ where $w$ is defined. 
\begin{align}
&w_t = \kappa w^2(t,x) \frac{x^{2s}}{(\kappa t)^2}\label{bcl2-def:w_t-igni}\\
&w_x =-2sw^2(t,x)\frac{x^{2s-1}}{\kappa t} \label{bcl2-def:w_x-igni}\\
&w_{xx} = 4s^2 w^{3}(t,x) \frac{x^{4s -2}}{(\kappa t)^2} + 2s(2-2s)w^2(t,x) \frac{x^{2s-2}}{\kappa t}  \label{bcl2-def:w_xx-igni}
\end{align}
Observe that since $s\leq 1/2$, we deduce from the latter that $w$ is convex in $[X(t),+\infty)$.  
Moreover by using the definition of $X(t)$ we also deduce that for $t\ge 1$,

\begin{equation}
w_x(t,X(t))  = \frac{-2s\eps \left[1-\gamma\eps\right]}{X(t)}.\label{bcl2-eq:esti-partial_w1-igni}
\end{equation}

\subsubsection{An estimate for $w$ on $[X(t)+1,+\infty)$.}

\begin{prop}\label{bcl2-prop:w2-igni}
For all $\kappa>0$ and all $\eps \in (0,1)$, $\gamma<\frac{1}{\eps}$,  we have for all $\, x\ge X(t)$,
\begin{align*}
 w(t,x)\le \frac{(1-\gamma\eps)\kappa t}{x^{2s}}.
\end{align*}
\end{prop}
\begin{proof}  
By using \eqref{bcl2-def:w-igni}, the definition of $w$, since we have 
$$
w(t,x)=  \frac{\kappa t }{x^{2s}} \left(1+\frac{\gamma \kappa t }{x^{2s}} \right)^{-1}
$$
Since $x\ge X(t)$,

$$
w(t,x)\le  \frac{\kappa t }{x^{2s}} \left(1+\frac{\gamma t\kappa}{X(t)^{2s}} \right)^{-1}\le  \frac{\kappa t (1-\eps\gamma)}{x^{2s}}.
$$

\end{proof}

%
%

\subsection{Estimating $\opd{\undu}$ when $x \leq X(t)$.}

On this region, by definition of $\undu$, we have 
\begin{equation*}
\opd{\undu}(t,x) = \int_{y \geq X(t)} [\undu(t,y) - \eps] J(x-y) \, dy.
\end{equation*}
This section aims at showing \eqref{eq:subsolleft-igni}. For the convenience of the reader, we shall state this is the following
\begin{prop}\label{bcl2-prop:esti-frac1-igni}
For all, $0<s\le \frac{1}{2},\, 1>\eps>  \theta, \gamma<\frac{1}{\eps}$ and $\kappa$ there exists $t_0(\eps,\kappa,s,\gamma)$ such that   for all $t\ge t_0$
$$ \opd{\undu}(t,x) + f(\eps) \ge 0 \quad \text{ for all }\quad x\le X(t).$$
\end{prop}

\begin{proof}

Recall that by \eqref{bcl2-eq:esti1}, we have 
\begin{equation*}
 \opd{\undu}\geq  - \frac{\eps\j_0}{2s}  \frac{1}{B^{2s}} -  \frac{3}{\eps}  \left(\j_1+\j_0 \int_{1}^{B} z^{1-2s}\, dz\right)w_x(t,X(t))^2.
\end{equation*}

We are now ready to choose $B:=\left(\frac{\eps(\j_0 +1)}{sf(\eps)}\right)^{\frac{1}{2s}}$ above. We then get 
$$
 \opd{\undu} \ge - \frac{f(\eps)}{2} -  \frac{3}{\eps}\left(\j_1+\j_0 \int_{1}^{B} z^{1-2s}\, dz\right)w_x(t,X(t))^2.
$$
Set for legibility $C_0:=3\left(\j_1+\j_0 \int_{1}^{B} z^{1-2s}\, dz\right),$ and use the estimate \eqref{bcl2-eq:esti-partial_w1-igni} on $ w_x(t,X(t))$, to get 
\begin{align*}
 \opd{\undu} + f(\eps)&\ge  \frac{f(\eps)}{2} -  \frac{C_0}{\eps} w_x(t,X(t))^2,\\
 &\ge \frac{f(\eps)}{2} -\frac{4C_0s^2\eps^2 \left[ 1-\gamma\eps \right]^2}{X^2(t)}.
\end{align*}
the proposition is proved by taking  $t$ large since $X(t)\to +\infty$. 

\end{proof}


%
%
%
%

\subsection{Estimate of $\opd{\undu}$ on $x>X(t)$.}

In this region, as exposed earlier and shown in \Cref{fig:zones-igni}, we shall estimate $\opd{\undu}$ in two separate intervals 
\begin{equation*}
[X(t), \sup\{X+R_0;X_\eta(t)\}],  \qquad [\sup\{X_\eta(t),X(t)+R_0\},+\infty ).
\end{equation*}
with $X_\eta$ to be chosen.

\subsubsection{The region $X(t)\le x\le \sup\left\{X(t)+R_0;X_{\eta}\right\}$}
For all $\eps>\theta$, let $\eta(\eps)>0$ be the smallest positive root of the polynomial function $z \mapsto \frac{1}{\eps^2}z^3+z -(\eps-\theta)$. Then
 for such $\eta$,
\begin{align*}
3(\eps-\eta)\left(1 -\frac{\eps -\eta}{\eps} +\frac{\eps -\eta}{3\eps^2} \right)= 3(\eps-\eta)\left(\frac{1}{3} +\frac{\eta}{3\eps} + \frac{\eta^2}{3\eps^2} \right)
 &= (\eps-\eta)\left(1 +\frac{\eta}{\eps}+ \frac{\eta^2}{\eps^2} \right)\\
&=\eps -\frac{\eta^3}{\eps^2}\\
&= \theta +\eta
\end{align*}
 
Let $X_\eta(t)$ be such that $w(t,X_{\eta}(t))=\eps - \eta$, then by construction we have $X_\eta>X(t)$ and from the above computation, $\undu(t,x)\ge \theta +\eta$ for all $x\le X_\eta(t)$.

We start this with an estimate
\begin{lemma}
For all $B>1$ and $\delta\ge \sup\{R_0,B+X(t)-x\}$,
\begin{equation}\label{bcl-eq:esti-frac2}
\opd{\undu} (t,x) \ge -\frac{\j_0\undu(t,x)}{s B^{2s}}-\frac{3}{\eps}\left(\j_1+\j_0 \int_{1}^{B} z^{1-2s}\, dz\right)\sup_{\stackrel{-B<\xi<B,}{x+\xi>X(t)}} \left( w_x(t,x+\xi)\right)^2.
\end{equation}
\end{lemma}

\begin{proof}
By definition of $\undu$, for any $\delta \ge R_0$ we have, using \Cref{hyp:J}, 
\begin{align}
\opd{\undu}(t,x)&= \int_{x+z \leq X(t)-\delta} \frac{\eps - \undu(t,x)}{ \vert  z\vert^{1+2s}} J(z)|z|^{1+2s} \, dz + \int_{x+z \geq X(t)-\delta} [\undu(t,x+z) - \undu(t,x)] J(z) \, dz,\nonumber\\
&= \frac{\j_0^{-1}}{2s} \frac{\eps - \undu(t,x)}{(x-X(t)+\delta)^{2s}} + \int_{x+z \geq X(t)-\delta} [\undu(t,x+z) - \undu(t,x)]J(z) \, dz.\label{bcl-def:frac}
\end{align}

We shall now estimate
$$\int_{x+z \geq X(t)-\delta} [\undu(t,x+z) - \undu(t,x)] J(z) \, dz.$$
For $B\ge 0$ to be chosen later on, we decompose,
\begin{multline}\label{bcl2-eq:esti2}
\int_{x+z \geq X(t)-\delta} [\undu(t,x+z) - \undu(t,x)] J(z) \, dz = \int_{x+z \geq X(t)-\delta, \vert z\vert \leq B} [\undu(t,x+z) - \undu(t,x)]J(z) \, dz \\+ \int_{x+z \geq X(t)-\delta,\vert  z\vert \geq B} [\undu(t,x+z) - \undu(t,x)] J(z) \, dz.
\end{multline}

The second integral in the right hand side of the above expression is the easiest. Since $\undu $ is positive and $J$ satisfies \eqref{hyp:J} we then have  for $B>1$, 
$$
\int_{x+z \geq X(t)-\delta, \vert  z\vert \geq B} [\undu(t,x+z) - \undu(t,x)]J(z) \, dz \geq - \undu(t,x)\j_0 \int_{x+z \geq X(t)-\delta, \vert  z\vert \geq B} \frac{dz}{ \vert  z\vert^{1+2s}}.$$

\noindent When  $ X(t)-\delta\le x-B$, a short computation shows that 
\begin{align*}
\int_{x+z \geq X(t)-\delta, \vert  z\vert \geq B} \frac{dz}{ \vert  z\vert^{1+2s}} &= \int_{X(t)-x-\delta \leq z \leq-B} \frac{dz}{ \vert  z\vert^{1+2s}} + \int_{z \geq B} \frac{dz}{ \vert  z\vert^{1+2s}}\\
&= \int_{X(t)-x-\delta \leq z \leq -B} \frac{dz}{ z^{1+2s}} + \int_{z \geq B} \frac{dz}{ z^{1+2s}}\\
&= \frac{1}{2sB^{2s}} - \frac{1}{2s(x+\delta -X(t))^{2s}}  +\frac{1}{2sB^{2s}}.
\end{align*}
On the other hand if $  X(t)-\delta\ge x -B$ then 
$$
\int_{x+z \geq X(t)-\delta, \vert  z\vert \geq B} \frac{dz}{ \vert  z\vert^{1+2s}} =  \int_{z \geq B} \frac{dz}{ \vert  z\vert^{1+2s}}=  \int_{z \geq B} \frac{dz}{ z^{1+2s}}= \frac{1}{2sB^{2s}}.
$$
In each situation we then have 

\begin{equation}\label{bcl2-eq:esti3}
\int_{x+z \geq X(t)-\delta, \vert  z\vert \geq B} [\undu(t,x+z) - \undu(t,x)] J(z) \, dz \geq -  \frac{\undu(t,x)\j_0}{sB^{2s}}.
\end{equation}

Let us now estimate the first integral of the right hand side of \eqref{bcl2-eq:esti2}, that is, let us estimate
$$ I:=\int_{x+z \geq X(t)-\delta, \vert  z\vert \leq B} [\undu(t,x+z) - \undu(t,x)]J(z) \, dz.$$ 
Since $\undu(t,x)$ is $\mathcal{C}^{1}$ in $x$ we have, for all $t\ge 1$ and $x\in \R$,
$$\undu(t,x+z)-\undu(t,x)=z\int_{0}^1 \undu_x(t,x+\tau z)d\tau $$ and therefore we  can rewrite $I$ as follows: 
$$I=\int_{x+z \geq X(t)-\delta, \vert  z\vert \leq B}\int_{0}^1 \undu_x(t,x+\tau z)z J(z)\,d\tau dz.$$
For any $\delta \ge B+X(t)-x$, let us observe that by  symmetry we have  
 $$\int_{x+z \geq X(t)-\delta, \vert  z\vert \leq B}\int_{0}^{1} J(z)z d\tau dz=0. $$
As a consequence we can rewrite $I$ as follows: 
$$I=\int_{x+z \geq X(t)-\delta, \vert  z\vert \leq B}\int_{0}^1  [\undu_x(t,x+\tau z)-\undu_x(t,x)]zJ(z)\,d\tau dz.$$
Since $\undu_x$ is a $\mathcal{C}^1$ function, by using the Taylor expansion 
$$ [\undu_x(t,x+\tau z)- \undu_x (t,x)]= \tau z\int_{0}^{1} \undu_{xx}(t,x+\tau\sigma z)d\sigma$$
we have 
\begin{align*}
I&=\int_{x+z \geq X(t)-\delta, \vert  z\vert \leq B}\int_{0}^1\int_{0}^1 \undu_{xx}(t,x+\sigma\tau z)\tau z^2 J(z)\,d\tau d\sigma dz,\\
&\ge \min_{-B<\xi<B}  \undu_{xx}(t,x+\xi)\left(\int_{\vert  z\vert \leq B}\int_{0}^1\int_{0}^1 \tau z^2 J(z)\,d\tau d\sigma dz \right)\\
&\ge \min_{-B<\xi<B} \undu_{xx}(t,x+\xi)\left(\int_{\vert  z\vert\le 1}\int_{0}^1\int_{0}^1 \tau z^2 J(z)\,d\tau d\sigma dz +\int_{1\le \vert  z\vert \leq B}\int_{0}^1\int_{0}^1 \tau z^2 J(z)\,d\tau d\sigma dz\right).
\end{align*}
By using \eqref{bcl2-def:u_tx} and \eqref{bcl2-def:u_xx}  and the convexity of $w$, we deduce that 
$$
I \ge -\frac{6}{\eps}\sup_{\stackrel{-B<\xi<B,}{x+\xi>X(t)}}  w_x(t,x+\xi)^2\left(\int_{\vert  z\vert\le 1}\int_{0}^1\int_{0}^1 \tau z^2 J(z)\,d\tau d\sigma dz +\int_{1\le \vert  z\vert \leq B}\int_{0}^1\int_{0}^1 \tau z^2 J(z)\,d\tau d\sigma dz\right).
$$
Hence we have 
\begin{equation}
I\ge  -\frac{3}{\eps} \left(\j_1 +  \j_0 \int_{1}^{B} z^{1-2s}\, dz\right) \sup_{\stackrel{-B<\xi<B,}{x+\xi>X(t)}} w_x(t,x+\xi)^2. \label{bcl2-eq:esti4} 
\end{equation}

Collecting \eqref{bcl2-eq:esti2}, \eqref{bcl2-eq:esti3} and \eqref{bcl2-eq:esti4}, we get the following estimate for all $B>1$ and $\delta\ge \sup\{R_0,B+X(t)-x\}$,
\begin{equation}
\opd{\undu} (t,x) \ge \frac{\j_0^{-1}}{2s} \frac{\eps - \undu(t,x)}{(x-X(t)+\delta)^{2s}}  -\frac{\j_0\undu(t,x)}{s B^{2s}}-\frac{3}{\eps}\left(\j_1+\j_0 \int_{1}^{B} z^{1-2s}\, dz\right)\sup_{\stackrel{-B<\xi<B,}{x+\xi>X(t)}} \left( w_x(t,x+\xi)\right)^2
\end{equation}
with ends the proof of the lemma since $\undu \leq \eps$.
\end{proof}

\begin{prop}\label{bcl2-prop:esti-frac2-igni}
For all $s\le \frac{1}{2}$,  $1>\eps>\theta, \kappa$ and $\gamma<\frac{1}{\eps}$  there exists $t_1$ such that 
$$\opd{\undu} + f(\undu) \ge \frac12 f(\undu)   \quad \text{ for all } \quad t\ge t_1, \quad X(t)<x<\sup\{X(t)+R_0;X_{\eta}\}.$$ 
\end{prop}

\begin{proof}
Recall that \eqref{bcl-eq:esti-frac2} is still all the way valid in this new context. Thus,  for any $\delta>\sup\{R_0, X(t)-x+B\}$, using that $\undu \leq \eps$, when $X(t)<x\le \sup\{X(t)+R_0;X_\eta(t)\}$,
\begin{equation*}
\opd{\undu} (t,x) \ge  -\frac{\j_0\eps}{s B^{2s}}-\frac{3}{\eps}\sup_{\stackrel{-B<\xi<B,}{x+\xi>X(t)}} \left(w_x(t,x+\xi)\right)^2\left(\j_1+ \j_0 \int_{1}^{B} z^{1-2s}\, dz\right).
\end{equation*}

Recall that since $x\le X_\eta(t)$, we have $\undu\ge \theta +\eta$.
On the other hand, observe that since   $x<X(t)+R_0$ and since $\undu$ is smooth we have 
\begin{align*}
\undu(t,x) \ge \undu(t,X(t)+R_0)&= \undu(t,X(t))+\undu(t,X(t)+R_0)-\undu(t,X(t))\\
 &= \undu(t,X(t))+R_0\int_{0}^1 \undu_x(t,X(t)+\tau)\,d\tau,
\end{align*}
and by using the definition of $\undu_x$ in \eqref{bcl2-def:u_t-u_x-igni}, and the convexity of $x \mapsto w(t,x)$ at any time, we deduce that  
$$ 
 \undu(t,x)\ge \undu(t,X(t))+3R_0 w_x(t,X(t))=\eps +3R_0 w_x(t,X(t)).
 $$
From the estimate   of   $w_x(t,X(t))$, \eqref{bcl2-eq:esti-partial_w1-igni}, it follows that 
$$
\undu(t,x)\ge \eps -\frac{6R_0s\eps [1-\gamma\eps]}{X(t)},
$$
which, thanks to $\lim_{t\to +\infty}X(t)=0$, enforces for $t\ge t'$,
$$
\undu(t,x)\ge \theta+\eta. 
$$
In both cases, we then have for $t\ge t'$,
$$
\undu(t,x)\ge \theta+\eta. 
$$

 As a consequence $f(\undu)>0$ for all $x\le \sup\{X(t)+R_0,X_\eta\}$ and $t>t'$.
Specify $B=\nu\frac{\j_0 \eps}{sf(\undu)}$ with $\nu>1$ to be chosen later on. Then from the above inequality we deduce that 

\begin{equation*}
\opd{\undu} (t,x) \ge  -\frac{f(\undu)}{\nu^{2s}}-\frac{3}{\eps}\sup_{\stackrel{-B<\xi<B,}{x+\xi>X(t)}} \left(w_x(t,x+\xi)\right)^2\left(\j_1 +\j_0 \int_{1}^{B} z^{1-2s}\, dz \right) \; \text{ when } \; x>X(t),
\end{equation*}
from which yields for $X(t)<x< X(t)+R_0$ and $t\ge t^*$,
$$\opd{\undu}(t,x)+f(\undu)\ge  f(\undu)\left(1-\frac{1}{\nu^{2s}}\right) -\frac{3}{\eps}\sup_{\stackrel{-B<\xi<B,}{x+\xi>X(t)}} \left(w_x(t,x+\xi)\right)^2\left(\j_1 +\j_0 \int_{1}^{\nu\eps^{\frac{1-\beta}{2s}}} z^{1-2s}\, dz\right).
$$
Recall that  $w$ is convex in $x$, so that
$$\sup_{\stackrel{-B<\xi<B,}{x+\xi>X(t)}} \left(w_x(t,x+\xi)\right)^2= w_x(t,X(t))^2,$$
and choose now $\nu>\nu_0:=\sup\left\{4^{\frac{1}{2s}}, \frac{s(f(\undu))}{\j_0 \eps}+1\right\}$, 
we then get using \eqref{bcl2-eq:esti-partial_w1-igni},
\begin{align*}
&\opd{\undu}+\frac{1}{2}f(\undu)\ge \frac{1}{4}f(\undu) -  4s^2\eps(1-\gamma\eps)^2C_1 \left(\frac{1}{X(t)}\right)^{2},
\end{align*}
with $C_1:=3\left[\j_1 + \j_0\int_{1}^{B} z^{1-2s}\, dz \right]$.

Finally recalling that $\ds{\lim_{t\to \infty}\frac{1}{X(t)} =0}$,  we may find $t_1\ge t^*$ such that for all $t\ge t_1$ the right hand side of the above expression is positive ending thus the proof of this proposition. 
\end{proof}

\subsubsection{The region $x>\sup\{X(t)+R_0;X_\eta\}$}

\begin{lemma}\label{lem:estDfar}
For any time $t >1$ and $x\ge \sup\{X(t)+R_0,X_2(t)\}$, 
\begin{equation} \label{bcl2-eq:esti-frac4}
\opd{\undu}(t,x)\ge  \frac{\eps - \undu(t,x)}{2s\j_0x^{2s}} +\j_1\min_{-1<\xi<1} \undu_{xx}(t,x+\xi)-\frac{\j_0\undu(t,x)}{2sB^{2s}}+\frac{3\j_0}{4} \left(\int_{1}^{B} z^{-2s} \, dz\right) \,w_x(t,x).
\end{equation}
\end{lemma}

\begin{proof}
 Let us go back to the definition of $\opd{\undu}(t,x)$ that we split into three parts:
$$
\opd{\undu}(t,x) = \int_{-\infty}^{-1} [u(t,x+z) - \undu(t,x)]J(z) \, dz + \int_{-1}^{1} [\undu(t,x+z) - \undu(t,x)]J(z) \, dz + \int^{\infty}_{1} [u(t,x+z) - \undu(t,x)] J(z) \, dz.$$
Since $x\ge X(t)+R_0$ and $\undu$ is decreasing, the first integral can be estimated as follows:   
\begin{align}
&\int_{-\infty}^{-1} [\undu(t,x+z) - \undu(t,x)] J(z) \, dz \nonumber\\
&\ge \j_0^{-1}\int_{-\infty}^{X(t)-x} \frac{\undu(t,x+z) - \undu(t,x)}{ \vert  z\vert^{1+2s}} \, dz +\int_{X(t)-x}^{-1} [\undu(t,x+z) - \undu(t,x)]J(z) \, dz\nonumber\\
&\ge  \frac{\j_0^{-1}}{2s} \frac{\eps - \undu(t,x)}{(x-X(t))^{2s}}.\label{bcl2-eq:esti5}
\end{align}
To obtain an estimate of the second integral, we actually follow the same steps as several times previously to obtain via Taylor expansion,  
%
%
%
\begin{align}
\int_{-1}^{1} [\undu(t,x+z) - \undu(t,x)]J(z) \, dz &=\int_{-1}^{1}\int_{0}^1\int_{0}^{1} \undu_{xx}(t,x+\tau \sigma z) \tau z^2J(z) \,  d\tau d\sigma dz \nonumber\\
&\ge \j_1\min_{-1<\xi<1} \undu_{xx}(t,x+\xi).\label{bcl2-eq:esti6}
\end{align}
Finally, let us estimate the last integral
\begin{align*}
I&:=\int_{1}^{+\infty}  [\undu(t,x+z) - \undu(t,x)]J(z)\,dz\\
&= \int_{1}^{B} [\undu(t,x+z) - \undu(t,x)] J(z) \, dz + \int_{B}^{+\infty} [\undu(t,x+z) - \undu(t,x)] J(z) \, dz,
\end{align*}
for $B> 1$ to be chosen later on. Since $\undu $ is positive we have  
\begin{equation}
\int_{B}^{\infty} [\undu(t,x+z) - \undu(t,x)]J(z)\,dz  \geq - \frac{\j_0\undu(t,x)}{2sB^{2s}}.\label{bcl2-eq:esti7}
\end{equation}
By using again a Taylor formula, the last integral rewrites
$$\int_{1}^{B} [\undu(t,x+z) - \undu(t,x)]J(z) \, dz=\int_{1}^{B}\int_{0}^1 \undu_x(t,x+\tau z) z J(z) \,d\tau dz. $$
Observe that since $x \geq X_2$ and $w$ is convex, \eqref{bcl2-def:u_x} implies
\begin{equation*}
\undu_x(t,x+\tau z) \geq \frac{3}{4} w_x(t,x). 
\end{equation*}
It follows that  
\begin{equation}\label{bcl2-eq:esti8}
\int_{1}^{B} [\undu(t,x+z) - \undu(t,x)]J(z) \, dz \geq \frac{3}{4} \left(\int_{1}^{B} z J(z) \, dz\right) \,w_x(t,x) \geq \frac{3\j_0}{4} \left(\int_{1}^{B} z^{-2s} \, dz\right) \,w_x(t,x).
\end{equation}
using \Cref{hyp:J}. Collecting \eqref{bcl2-eq:esti5}, \eqref{bcl2-eq:esti6},\eqref{bcl2-eq:esti7}, \eqref{bcl2-eq:esti8}, we find for $x\ge X(t)+R_0$,
\begin{equation*} 
\opd{\undu}(t,x)\ge  \frac{\eps - \undu(t,x)}{2s\j_0x^{2s}} +\j_1\min_{-1<\xi<1} \undu_{xx}(t,x+\xi)-\frac{\j_0\undu(t,x)}{2sB^{2s}}+\frac{3\j_0}{4} \left(\int_{1}^{B} z^{-2s} \, dz\right) \,w_x(t,x),
\end{equation*}
which ends the proof of the lemma.

%
\end{proof}

Recall that by \Cref{lem:estDfar}, we have in the range $x\ge X(t)+R_0$,
\begin{equation} \label{bcl2-eq:esti-frac4-igni}
\opd{\undu}(t,x)\ge  \frac{\eps - \undu(t,x)}{2s\j_0x^{2s}} +\j_1\min_{-1<\xi<1} \undu_{xx}(t,x+\xi)-\frac{\j_0\undu(t,x)}{2sB^{2s}}+\frac{3\j_0}{4} \left(\int_{1}^{B} z^{-2s} \, dz\right) \,w_x(t,x).
\end{equation}
Let us now estimate $\opd{\undu}$ when $x\ge \sup\{X(t)+R_0,X_\eta(t)\}$. 

\begin{prop}\label{bcl2-prop:esti-frac3-igni}
For any $s<\frac{1}{2}$, $\kappa>0$, $\eps >\theta$ there exists $\gamma_0$ and $t_3 > 0$ such that for all $t\ge t_3$ and $\gamma_0\le \gamma<\frac{1}{\eps}$
$$\opd{\undu}(t,x) \ge \frac{\eta}{16\j_0sx^{2s}} \quad \text{ for  all } \quad x>\sup\{X(t)+R_0,X_\eta(t)\}.$$
\end{prop}

\begin{proof}
Let us recall that $X_\eta(t)$ is such that $w(t,X_\eta(t))=\eps -\eta$ and consider $x\ge \sup\{X_\eta(t),X(t)+R_0\}$. For such $x$, thanks to Proposition \ref{bcl2-prop:w2-igni}, we have $\undu(t,x) \le w(t,x)\le [1-\gamma\eps]\frac{\kappa t}{x^{2s}}$ and therefore  we have 

$$
\opd{\undu}(t,x)\ge  \frac{\eta}{2s\j_0x^{2s}}\left(1-\frac{\j_0^2\kappa t(1-\eps\gamma)}{\eta B^{2s}}\right) +\j_1\min_{-1<\xi<1} \undu_{xx}(t,x+\xi)+3\j_0\int_{1}^{B}z^{-2s}\,dz  w_x(t,x).
$$

Let us rewrite $\gamma:=\frac{1-\sigma}{\eps}$ with $\sigma\in (0,1)$, then we have 

$$
\opd{\undu}(t,x)\ge  \frac{\eta}{2s\j_0x^{2s}}\left(1-\frac{\j_0^2\kappa t \sigma}{\eta B^{2s}}\right) +\j_1\min_{-1<\xi<1} \undu_{xx}(t,x+\xi)+3\j_0\int_{1}^{B}z^{-2s}\,dz  w_x(t,x).
$$

Let us now estimate from below $\ds{\min_{-1<\xi<1} \undu_{xx}(t,x+\xi)}$.  Using \eqref{bcl2-def:u_xx-igni} and the convexity of $w$, we see that 
$$\min_{-1<\xi<1} \undu_{xx}(t,x+\xi)\ge -\frac{6}{\eps} \left(w_x(t,x-1)\right)^2,$$
which thanks to \eqref{bcl2-def:w_x-igni} and that $w(t,x-1)\le w(t,X(t))=\eps$ leads to 
$$ \min_{-1<\xi<1} \undu_{xx}(t,x+\xi)\ge -24s^2\eps w^{2}(t,x-1)\frac{(x-1)^{4s-2}}{(\kappa t)^{2}}.$$
By proposition \ref{bcl2-prop:w2-igni} since $X(t)+1\le X(t)+R_0\le x$, 
we have 
$$ w(t,x-1)\le \frac{\kappa t\sigma}{(x-1)^{2s}}, $$
and thus since $x\ge X_\eta(t)$
$$ \min_{-1<\xi<1} \undu_{xx}(t,x+\xi)\ge -24s^2\eps\sigma^2\frac{1}{(x-1)^2}\ge -48s^2\eps\sigma^2\frac{1}{(X_\eta(t)-1)^{2-2s}}\frac{1}{x^{2s}}.$$

Now by using that $ s\le \frac12$, the definition of $X_\eta(t)$, since $X_{\eta}(t)>(\kappa t)^{\frac{1}{2s}}\left(\frac{\eta}{\eps^2}\right)^{\frac{1}{2s}}$  we may find $t'(\kappa,s,\eps)$ independent of $\sigma$ such that for all $t\ge t'$ 
$$
\j_1\min_{-1<\xi<1} \undu_{xx}(t,x+\xi)\ge -\frac{\eta}{8\j_0sx^{2s}}. 
$$
 
Therefore for all $t\ge t'$ we then get 

$$
\opd{\undu}(t,x)\ge  \frac{\eta}{2s\j_0x^{2s}}\left(\frac{3}{4}-\frac{\j_0^2\kappa t \sigma}{\eta B^{2s}}\right)+3\j_0\int_{1}^{B}z^{-2s}\,dz  w_x(t,x).
$$

Again, by using Proposition \ref{bcl2-prop:w2-igni}, we deduce that 
$$w_x(t,x)=-2s\frac{x^{2s-1}}{\kappa t}w^{2}(t,x)\ge -2s\frac{(\kappa t)^2\sigma^2}{x^{4s}}\frac{x^{2s-1}}{\kappa t}=-\frac{\kappa t}{x^{2s+1}}\left(2s\sigma^2\right),  $$

and we then have 

$$
\opd{\undu}(t,x)\ge  \frac{\eta_0}{2s\j_0x^{2s}}\left(\frac{3}{4}-\frac{\j_0^2\kappa t \sigma}{\eta B^{2s}} -\frac{12 s^2\j_0^2\sigma^2\kappa t}{\eta x}\int_1^B z^{-2s}\,dz\right).
$$

Let $C_1:=\frac{\j_0^2}{\eta}$ and $C_2:=12s^2$, we then have 

$$
\opd{\undu}(t,x)\ge  \frac{\eta}{2s\j_0x^{2s}}\left(\frac{3}{4}-C_1\kappa t\sigma\left[\frac{1}{B^{2s}} +\frac{C_2\sigma}{x}\int_{1}^B z^{-2s}\,dz\right]\right).
$$

We now treat the case $s=\frac{1}{2}$ and $s<\frac{1}{2}$

Recall now that $s<\frac{1}{2}$, so for $B>1$ we have 

$$\int_{1}^B z^{-2s}\,dz\le \frac{B^{1-2s}}{1-2s},$$
and the above expression reduce to  

$$
\opd{\undu}(t,x)\ge  \frac{\eta}{2s\j_0x^{2s}}\left(\frac{3}{4}-C_1\kappa t\sigma\left[\frac{1}{B^{2s}} +\frac{\mathfrak{C}_2\sigma}{x}B^{1-2s}\right]\right).
$$
with $\mathfrak{C}_2:=\frac{C_2}{1-2s}$.

Let us now take $B=\frac{2s x}{(1-2s)C_2\sigma}$ and check whether $B>1$. 
Since $x>X_\eta(t)$, this means that $$ B\ge \frac{2s}{\mathfrak{C}_2(1-2s)\sigma}(\kappa t)^{\frac{1}{2s}}\left(\frac{\eta}{\eps^2}\right)^{\frac{1}{2s}},$$ 
so for $t\ge t^{"}:=\left(\frac{\mathfrak{C}_2(1-2s)\sigma}{s}\right)^{2s} \kappa\left( \frac{\eps^2}{\eta}\right)$
we then have $B>2$ and a short computations shows  
 
\begin{align*}
\left[\frac{1}{B^{2s}} +\frac{\mathfrak{C}_2\sigma}{x}B^{1-2s}\right]&=\frac{(1-2s)^{2s}\mathfrak{C}_2^{2s}\sigma^{2s}}{(2s x)^{2s}} + \frac{\mathfrak{C}_2\sigma}{x}\frac{(2sx)^{1-2s}}{\sigma^{1-2s}\mathfrak{C}_2^{1-2s} (1-2s)^{1-2s}}\\
&=\frac{(1-2s)^{2s}\mathfrak{C}_2^{2s}\sigma^{2s}}{(2s x)^{2s}} + \frac{\mathfrak{C}_2^{2s}\sigma^{2s}(2s)^{1-2s}}{x^{2s} (1-2s)^{2-2s}}\\
&=\frac{\sigma^{2s}}{x^{2s}} C_3
\end{align*}
with $C_3:= \mathfrak{C}_2^{2s}\left( \frac{(1-2s)^{2s}}{(2s)^{2s}}+\frac{(2s)^{1-2s}}{(1-2s)^{1-2s}}\right)$.

As a consequence, we have 

$$
\opd{\undu}(t,x)\ge  \frac{\eta}{2s\j_0x^{2s}}\left(\frac{3}{4}-\frac{C_1C_3\kappa t\sigma^{1+2s}}{x^{2s}}\right).
$$

Since $x\ge X(t)$, by exploiting \eqref{bcl2-def:X-igni}, the definition of $X(t)$, we then achieve

\begin{align*}
\opd{\undu}(t,x)&\ge  \frac{\eta}{2s\j_0x^{2s}}\left(\frac{3}{4}-\frac{C_1C_3\kappa t\sigma^{1+2s}}{X^{2s}(t)}\right)\\
&\ge\frac{\eta}{2s\j_0x^{2s}}\left(\frac{3}{4}-C_1C_3\sigma^{2s}\right)
\end{align*}

and the proposition is then proved by taking $\sigma$ small and $t\ge t_3:=\sup\{t',t^{"}\}$.

%
%
%
%
\end{proof}

\subsection{Tuning the parameter $\kappa$}

In this last part of the proof, we choose our parameter $\kappa$ in order that for some $t^*>0$, $\undu$ is indeed a sub-solution to \eqref{eq:main} for  $t\ge t^*$. 
Recall that  $\undu$ is a subsolution if and only if \eqref{eq:subsolleft-igni} and \eqref{eq:subsolright-igni} hold simultaneously. Since \eqref{eq:subsolleft-igni} holds unconditionally for $t$ sufficiently large, the only thing left to check is that \eqref{eq:subsolright-igni} holds for a suitable choice of  $\kappa$.

By using \eqref{bcl2-def:u_t-u_x-igni} and \eqref{bcl2-def:w_t-igni}, \eqref{eq:subsolright-igni} holds if particular
\begin{align*}
3\frac{x^{2s}}{\kappa t^2}w^{2}(t,x)&\leq \opd{\undu}(t,x) + f(\undu), &x > X(t),
\end{align*}

Set $t^*:=\sup\{t_0,t_1,t_2,t_3\}$, where $t_0,t_1,t_2$ and $t_3$ are respectively determined by \Cref{bcl2-prop:esti-frac1-igni}, \Cref{bcl2-prop:esti-frac2-igni} and \Cref{bcl2-prop:esti-frac3-igni}. To make our choice, let us decompose the set $[X(t),+\infty)=I_1 +I_2$ into two subsets defined as follows
$$I_1:= [X(t),X_\eta(t)],\qquad I_2:=[X_\eta(t),+\infty).$$

On the first interval, we have
\begin{lemma}
For all $0<s\le \frac{1}{2},$ $\eps>\theta$, there exists  $t_4(\eps)\ge t^*$ such that for all $\kappa$, one has
$$3\frac{x^{2s}}{\kappa t^2}w^{2}(t,x)\leq \opd{\undu} + f(\undu),\quad \text{for all }\quad x \in  I_1.$$
\end{lemma}

\begin{proof}
First observe that  since $x\ge X_\eta$, by the previous proofs, we know that $\undu>\theta +\eta$ and thus $f(\undu)>\min_{s\in[\theta +\eta_0,\eps]}f(s)=:m_0>0$ for all  $t\ge t^*$ and $x \in I_1$.

Now by exploiting the definition of $X_{\eta}(t)$ and $X(t)$ it follows that   $x\le \left(\kappa t \left[1+\frac{1}{\eps -\eta}\right]\right)^{\frac{1}{2s}}$, and that  
$$3\frac{x^{2s}}{\kappa t^2}w^{2}(t,x)\le  \frac{3}{t}\eps^2\left[1+\frac{1}{\eps -\eta}\right],$$
which by taking $t$ large then yields, says $t\ge t'$,
$$3\frac{x^{2s}}{\kappa t^2}w^{2}(t,x)\le  \frac{f(\undu)}{2}.$$

Recall that by \Cref{bcl2-prop:esti-frac2-igni}, we have for all $x\in I_1$ and $t\ge t^*$ 
$$
\opd{\undu} + f(\undu)\ge \frac{f(\undu)}{2}.$$
We then end our proof by taking $\gamma^*:=\inf\{\gamma_0,\frac{1-\eps}{4}\}$ and $t\ge t_4:=\sup\{t',t^*\}$.
\end{proof}

Finally, let us check what happens on $I_2$, 
\begin{claim}
There exists $\kappa^*$ and $t_5\ge t^*$ such that for all  $\kappa\le \kappa^*$,
$$3\frac{x^{2s}}{\kappa t^2}w^{2}(t,x)\leq \opd{\undu} + f(\undu),\quad \text{for all }\quad x \in  I_2.$$
\end{claim}

\begin{proof}

By Proposition \ref{bcl2-prop:w2-igni}, we have for $x\in I_2$ and $t\ge t^*$ 
$$ w(t,x)\le \frac{\kappa t}{x^{2s}}[1+\eps],$$
Therefore, we have 
 $$3\frac{x^{2s}}{\kappa t^2}w^{2}(t,x)\le 3\frac{x^{2s}}{\kappa t^2} \frac{(\kappa t)^2}{x^{4s}}[1+\eps]^2=\frac{\kappa[1+\eps]^2}{x^{2s}}.$$
 Now recall that by  Proposition \ref{bcl2-prop:esti-frac1-igni}, we have for all $x\in I_2$ and $t\ge t^*$ 
$$
\opd{\undu} + f(\undu)\ge \frac{\eta}{16\j_0s x^{2s}}.$$

The claim is then proved by taking $\kappa \le \kappa^*:=\frac{1}{48 \j_0 s[1+\eps]^2}$.
\end{proof}

\bibliographystyle{abbrv}
\bibliography{biblio}

\end{document}